\documentclass{amsart}%

\usepackage{hyperref}
\usepackage{cite}
\usepackage{amssymb}
\usepackage{amsmath}
\usepackage{amsthm}
\usepackage{amsfonts}
\usepackage{graphicx}
\usepackage{bbm}
\usepackage{color}
\usepackage[toc,page]{appendix}
\usepackage{subfigure}
\usepackage{mathtools}
\usepackage{esint}
\usepackage{enumitem}

\newcommand{\beq}{\begin{equation}}
\newcommand{\eeq}{\end{equation}}
\newcommand{\bea}{\begin{eqnarray}}
\newcommand{\eea}{\end{eqnarray}}
\newcommand{\beas}{\begin{eqnarray*}}
\newcommand{\eeas}{\end{eqnarray*}}

\newtheorem{theorem}{Theorem}[section]

\newtheorem{assumption}[theorem]{Assumption}

\newtheorem{remark}[theorem]{Remark}
\newtheorem{example}[theorem]{Example}
\newtheorem{examples}[theorem]{Examples}
\newtheorem{foo}[theorem]{Remarks}

\newtheorem*{acknowledgement}{Acknowledgement}
\newcommand{\M}{\mathbb M}

\begin{document}

\title[Gradient bounds]{On the Cheng-Yau gradient estimate for Carnot groups and sub-Riemannian manifolds}

\author[Baudoin]{Fabrice Baudoin{$^{\star}$}}
\thanks{\footnotemark {$\star$} This research was supported in part by NSF Grant DMS-1660031.}
\address{Department of Mathematics\\
University of Connecticut\\
Storrs, CT 06269, USA} \email{fabrice.baudoin@uconn.edu}

\author[Gordina]{Maria Gordina{$^{\dagger\ddag}$}}
\thanks{\footnotemark {$\ddag$} Research was supported in part by the Simons Fellowship.}
\thanks{\footnotemark {$\dagger $} This research was supported in part by NSF Grants DMS-1405169, DMS-1712427}
\address{Department of Mathematics\\
University of Connecticut\\
Storrs, CT 06269, USA} \email{maria.gordina@uconn.edu}

\author[Mariano]{Phanuel Mariano{$^{\dagger}$}}
\address{Department of Mathematics\\
		Purdue University\\
		West Lafayette, IN 47907,  USA} \email{pmariano@purdue.edu}


\begin{abstract}
In this note we show how results in \cite{BaudoinBonnefont2016, BaudoinGarofalo2013, CoulhonJiangKoskelaSikora2017} yield the Cheng-Yau estimate on two classes of sub-Riemannian manifolds: Carnot groups and sub-Riemannian manifolds satisfying a generalized curvature-dimension inequality.
\end{abstract}


\maketitle

\tableofcontents


\renewcommand{\contentsname}{Table of Contents}

\section{Introduction}

Let $M$ be a $d$-dimensional Riemannian complete non-compact manifold with the Ricci curvature bounded below by $-(d-1)K$. Let $u$ be a positive harmonic function in a Riemannian ball $B(x_{0}, 2r)$, then we say that $u$ satisfies the Cheng-Yau estimate if

\begin{equation}\label{Riemannian-CY}
\sup_{B(x_{0},r)}\left|\nabla\log u\right|\leqslant C_{d}\left(\frac{1}{r}+\sqrt{K}\right),
\end{equation}
where $C_{d}$ is a global constant depending only on the dimension $d$. In particular, when $K=0$ this estimate shows that positive harmonic functions are constant. This estimate was formulated in a more general form in \cite{Yau1975a, ChengYau1975a}, and stated as in \eqref{Riemannian-CY} in \cite[Theorem 3.1]{SchoenYau1994}. Sharp versions of the Cheng-Yau inequality were given in \cite{Li-Wang2002,Munteanu2012}. The stability of \eqref{Riemannian-CY} under certain perturbations of the metric was considered in \cite{Zhang2006}

The standard curvature arguments are not easily available in the case when $M$ is replaced by a sub-Riemannian manifold.  Nevertheless, there has been significant progress in geometric analysis on sub-Riemannian manifolds in \cite{BakryBaudoinBonnefontChafai2008, BaudoinBonnefont2016, BaudoinBonnefontGarofalo2014, BaudoinGarofalo2017}. Even in the absence of a Riemannian structure,  \cite{BaudoinGarofalo2017} developed new techniques to prove a number of results which in the Riemannian setting go back to the work of Yau and Li-Yau.  The main tool in \cite{BaudoinGarofalo2017} relied on a generalized curvature-dimension inequality on a class of sub-Riemannian manifolds with transverse symmetries.

Carnot groups also form a large and interesting class of sub-Riemannian manifolds. These are Lie groups whose Lie algebra admits a stratified structure. This stratified structure allows for H\"ormander's condition \cite[Theorem 1.1]{Hormander1967a} to be satisfied. H\"ormander's theorem guarantees that the sub-Laplacian associated with the structure of a Carnot group is hypoelliptic. In particular, this gives us the existence of a smooth heat kernel for this Laplacian. Most Carnot groups do not satisfy a generalized curvature-dimension inequality, so one needs to employ different techniques than in \cite{BaudoinGarofalo2017}.

The Cheng-Yau estimate was proved in \cite[Corollary 4.6]{BanerjeeGordinaMariano2018} for the simplest non-commutative Carnot group, the Heisenberg group, using probabilistic (coupling) techniques. The purpose of this note is to show that this estimate can be proven on two classes of sub-Riemannian manifolds, namely, sub-Riemannian manifolds satisfying a generalized curvature-dimension inequality  and Carnot groups, by relying on results from \cite{CoulhonJiangKoskelaSikora2017, BaudoinBonnefont2016, BaudoinGarofalo2013}. In particular, this recovers the known fact that global non-negative harmonic functions
in these two settings have to be constant (see \cite[Theorem 5.8.1]{BonfiglioliLanconelliUguzzoniBook} and \cite[Theorem 5.1]{BaudoinBonnefontGarofalo2014} in more generality).

\section{Carnot groups}
\subsection{Preliminaries}

We recall that a Carnot group of step $N$ is a simply connected Lie group $\mathbb{G}$ whose Lie algebra can be written as
\[
\mathfrak{g}=\mathcal{V}_{1}\oplus\cdots\oplus\mathcal{V}_{N},
\]
where
\[
\left[\mathcal{V}_{i},\mathcal{V}_{j}\right]=\mathcal{V}_{i+j}
\]
and $\mathcal{V}_{k}=0$ for $k>N$. In particular, Carnot groups are nilpotent.

Let $V_{1}, \dots, V_{d}$ be a linear basis for the vector space $\mathcal{V}_{1}$. The $V_i$s can be viewed as left-invariant vector fields on $\mathbb{G}$. The left-invariant sub-Laplacian on $\mathbb{G}$ is the operator

\begin{equation}\label{subL-Carnot}
L=\sum_{i=1}^{d}V_{i}^{2}.
\end{equation}
Let $\mu$ be the bi-invariant Haar measure on $\mathbb{G}$. Since Carnot groups are complete  the operator $L$ in \eqref{subL-Carnot} is essentially self-adjoint on $L^2\left(\mathbb{G},\mu\right)$, with domain being the space of smooth and compactly supported functions $f:\mathbb{G}\to\mathbb{R}$ denoted by $C_c^\infty \left(\mathbb{G} \right)$. We abuse notation and denote by $L$ the Friedrichs extension of this operator to a unique non-positive self-adjoint operator on $L^2\left(\mathbb{G},\mu\right)$. Then the heat semigroup $\left(P_{t}\right)_{t\geqslant 0}$ on $\mathbb{G}$ can be defined through the spectral theorem. As $L$ is hypoelliptic, $P_{t}$ admits a positive smooth fundamental solution to the heat equation called the heat kernel $p_{t}\left( g,g^{\prime} \right)$.

Denote by $\nabla = \left( V_1,\dots,V_d \right)$ the gradient determined by this basis, and denote by $\Vert \cdot \Vert$ the usual Euclidean norm. The \emph{carr\'e du champ}
operator of $L$ is defined by
\[
\Gamma(f,f):=\frac{1}{2}\left(Lf^2-2fLf \right) = \left\Vert \nabla f\right\Vert ^{2}=\sum_{i=1}^{d}\left(V_{i}f\right)^{2},
\]
and it is often thought of as the square of the \emph{length of the gradient $\nabla$}. We let $d$ be the  Carnot-Carath\'eodory distance on $\mathbb{G}$ making  $\left(\mathbb{G}, d\right)$ a metric space. We refer the reader to \cite{BonfiglioliLanconelliUguzzoniBook} for more details and results on Carnot groups.

\subsection{The Cheng-Yau estimate}

We say a function $u:\mathbb{G}\to \mathbb{R}$ is \emph{harmonic} in a domain $D\subset \mathbb{G}$ if $Lu=0$ on $D\subset \mathbb{G}$.

\begin{theorem}\label{CarnotCY}
If $u$ is any positive harmonic function for $L$ in a ball $B\left(x,2r\right)\subset\mathbb{G}$,
then there exists a constant $C>0$ not dependent on $u, r$
and $x$ such that
\begin{equation}\label{eq:CY}
\sup_{B(x,r)}\left\Vert \nabla\log u \right\Vert \leqslant\frac{C}{r}.
\end{equation}
Moreover, if $u$ is a positive harmonic function on $\mathbb{G}$, then $u$ must be equal to a constant.
\end{theorem}

\begin{proof}

In the proof, $C$ will denote a generic positive constant that does not depend on $u, r$ and $x$ whose value might change from line to line. First recall the reverse Poincar\'e inequality for the heat semigroup obtained for Carnot groups in \cite[Proposition 2.5]{BaudoinBonnefont2016}
\begin{equation}\label{eq:ReversePoincare}
\left\Vert \nabla P_{t}f(g)\right\Vert^{2} \leqslant \frac{C}{t}\left(P_{t}f^{2}(g)-\left(P_{t}f\right)^{2}(g)\right)
\end{equation}
for functions $f\in C_c^\infty \left(\mathbb{G}\right)$. Note that for functions $f\in L^{\infty}\left(\mathbb{G}\right)$ we have
\begin{equation}\label{Eq2}
P_{t}f^{2}(g) \leqslant \left\Vert f\right\Vert _{L^{\infty}(G)}^{2},
\end{equation}
therefore combining \eqref{eq:ReversePoincare} and \eqref{Eq2} we obtain

\begin{equation}\label{Eq3}
\left\Vert \nabla P_{t}f(g)\right\Vert ^{2} \leqslant \frac{C}{t}\left\Vert f\right\Vert_{L^{\infty}(G)}^{2}.
\end{equation}
Taking a square root in \eqref{Eq3} implies that
\begin{equation}
\left\Vert \nabla P_{t}\right\Vert _{\infty\to\infty}\leqslant\frac{C}{\sqrt{t}}.\label{eq:Boundness}
\end{equation}
Applying \cite[Theorem 1.2]{CoulhonJiangKoskelaSikora2017}, in particular that $(iii)$ implies $(i)$,
shows that  \eqref{eq:Boundness} implies that there exists a $C>0$
such that, for every ball $B(x,r)$ and every function
$u$ that is harmonic in $B(x,2r)$ we have

\begin{equation}\label{eq:RHinfty}
\left\Vert \nabla u\right\Vert_{L^{\infty}(B(x,r))} \leqslant \frac{C}{r\mu(B(x,2r))}\int_{B(x,2r)}\left|u\right|d\mu.
\end{equation}
Then we can apply \cite[Lemma 2.3]{CoulhonJiangKoskelaSikora2017} to show that \eqref{eq:RHinfty}
implies the Cheng-Yau estimate \eqref{eq:CY}.

We rely on results in \cite{CoulhonJiangKoskelaSikora2017} that require several assumptions that are satisfied for Carnot groups as follows. Their first assumption is that the underlying space is a non-compact doubling Dirichlet metric measure space. The space $\mathbb{G}$ is doubling since by \cite[Proposition 11.15]{HajlaszKoskela2000} (or more classically by \cite{Karidi1994,FollandSteinBook1982}) there exists a $C>0$ independent of $x\in\mathbb{G} , r>0$ such that $\left|B(x,r)\right|=Cr^{Q}$ where $Q=\sum_{j=1}^N j \dim \left(\mathcal{V}_j\right)$ is the homogeneous dimension of the $\mathbb{G}$. Here $\left|E\right|$ denotes the Lebesgue measure of the set $E$ and recall that the Haar measure $\mu$ is Lebesgue measure up to a constant. We also have that $\mathbb{G}$ is Dirichlet space as described in \cite[Section 3, pp.233-234]{Sturm1995b}. This is actually true in general for H\"ormander's type operators with bounded measurable coefficients on Lie groups having polynomial volume growth in the sense of \cite{Saloff-CosteStroock1991}. Another ingredient in \cite{CoulhonJiangKoskelaSikora2017} is upper Gaussian bounds on the heat kernel which follow from \cite[Theorem VIII2.9]{VaropoulosBook1992}. The space also supports a local scale-invariant $L^{2}$-Poincar\'e inequality by \cite[Proposition 11.17]{HajlaszKoskela2000}.

Finally, if $u$ is a positive harmonic function on all of $\mathbb{G}$, taking $r\to\infty$ in \eqref{eq:CY} gives us that $u$ must be constant.

\end{proof}

\section{Sub-Riemannian manifolds}

\subsection{Preliminaries}
In this section we study the setting similar to \cite{BaudoinBonnefontGarofalo2014}. We state relevant details here for completeness. Let $\left(\mathbb{M}, \mu\right)$ be a measure space, where $\mathbb{M}$ is an $n$-dimensional $C^{\infty}$ connected manifold endowed with a smooth measure $\mu$.  Recall (e.g. \cite[p.85]{Tao5thHilbertBook}) that the measure  $\mu$ on a smooth  manifold $\mathbb{M}$ is called a \emph{smooth measure} if $\mu$ is a Radon measure which has a smooth Radon-Nikodym derivative with respect to the Lebesgue measure when viewed in coordinates, that is, for any smooth coordinate chart $\varphi: U \longrightarrow V$, $U\subset M$, $V \subset \mathbb{R}^{n}$, the pushforward measure $\varphi_{\ast}\left( \mu \right)$ has a smooth Radon-Nikodym derivative with respect to the Lebesgue measure on $V$. Let $L$ be a second order diffusion operator on $\mathbb{M}$ which is locally subelliptic (in the sense of  \cite{FeffermanSanchez-Calle1986, JerisonSanchez-Calle1987}). We refer to \cite[Section 1]{BaudoinLevico2018} for a detailed account on properties of locally  subelliptic operators and associated distances that we are going to use in the sequel. In addition, we assume that
\begin{align*}
 & L1=0,
 \\
 & \int_{\mathbb{M}}fLgd\mu=\int_{\mathbb{M}}gLfd\mu,
 \\
 & \int_{\mathbb{M}}fLf \leqslant 0
\end{align*}
for every $f, g \in C_{c}^{\infty}\left(\mathbb{M}\right)$, where as before $C_{c}^{\infty}\left(\mathbb{M}\right)$ denotes the space of smooth compactly supported functions on $\mathbb{M}$.

The space $\mathbb{M}$ is endowed with a \emph{carr\'e du champ} operator
defined by
\[
\Gamma(f,g):=\frac{1}{2}\left(L\left(fg\right)-fLg-gLf\right), \hskip0.05in f, g\in C^{\infty}\left(\mathbb{M}\right).
\]
We denote $\Gamma(f)=\Gamma(f,f)$. It is not too hard to see that $\Gamma(f)\geqslant 0$ for all $f\in C^{\infty}\left(\mathbb{M}\right)$. 
We will also assume the existence of a symmetric, first-order differential
bilinear form $\Gamma^{Z}:C^{\infty}\left(\mathbb{M}\right)\times C^{\infty}\left(\mathbb{M}\right)\to C^{\infty}\left(\mathbb{M}\right)$
that satisfies
\begin{align*}
\Gamma^{Z}\left(fg,h\right) & =f\Gamma^{Z}\left(g,h\right)+g\Gamma^{Z}\left(f,h\right),\\
\Gamma^{Z}(f) & =\Gamma^{Z}(f,f)\geqslant0,
\end{align*}
for all $f,g,h\in C^{\infty}\left(\mathbb{M}\right)$.
Given the first order bi-linear forms $\Gamma$
and $\Gamma^{Z}$ on $\mathbb{M}$, we can introduce the following
second-order differential forms
\[
\Gamma_{2}(f,g)=\frac{1}{2}\left(L\Gamma(f,g)-\Gamma\left(f,Lg\right)-\Gamma\left(g,Lf\right)\right)
\]
and
\[
\Gamma_{2}^{Z}(f,g)=\frac{1}{2}\left(L\Gamma^{Z}(f,g)-\Gamma^{Z}\left(f,Lg\right)-\Gamma^{Z}\left(g,Lf\right)\right).
\]
Similar to $\Gamma$, we will use the notation $\Gamma_{2}(f):=\Gamma_{2}(f,f),\Gamma_{2}^{Z}(f):=\Gamma_{2}^{Z}(f,f)$.

We suppose the following assumptions to hold throughout this section.
\begin{enumerate}[label=(\Roman*)]
\item \label{A.1} There exists an increasing sequence $h_{k}\in C_{c}^{\infty}\left(\mathbb{M}\right)$
such that $h_{k}\uparrow1$ on $\mathbb{M}$, and
\[
\left\Vert \Gamma\left(h_{k}\right)\right\Vert _{\infty}+\left\Vert \Gamma^{Z}\left(h_{k}\right)\right\Vert _{\infty}\to0,\,\,\,\mbox{as }k\to\infty.
\]

\item\label{A.2} For any $f\in C^{\infty}\left(\mathbb{M}\right)$ one has
\[
\Gamma\left(f,\Gamma^{Z}\left(f\right)\right)=\Gamma^{Z}\left(f,\Gamma\left(f\right)\right).
\]

\item\label{A.3} The \emph{generalized curvature-dimension inequality} $CD\left(\rho_{1},\rho_{2},\kappa,d\right)$ is satisfied with $\rho_{1}\geqslant0$. That is, there exist constants $\rho_{1}\geqslant 0,\rho_{2}>0,\kappa\geqslant0$, and $d \geqslant 2$ such that the following inequality holds
\[
\Gamma_{2}\left(f\right)+\nu\Gamma_{2}^{Z}\left(f\right) \geqslant \frac{1}{d}\left(Lf\right)^{2}+\left(\rho_{1}-\frac{\kappa}{\nu}\right)\Gamma\left(f\right)+\rho_{2}\Gamma^{Z}\left(f\right),
\]
for all $f\in C^{\infty}\left(\mathbb{M}\right)$ and every $\nu>0$.

\item \label{A.4} The heat semigroup generated by $L$, which will be denoted $P_t$, is stochastically complete, that is, for $t\geqslant 0, P_t 1=1$ and for every $f\in C_c^\infty(\mathbb{M})$ and $T \geqslant 0$, one has
\[
\sup_{t\in [0,T]}\left\Vert\Gamma \left(P_t f \right)  \right\Vert_\infty + \left\Vert\Gamma^Z\left(P_t f \right)  \right\Vert_\infty < +\infty.
\]
\item \label{A.5} Given any two points $x,y\in \mathbb{M}$, there exists a subunit curve (in the sense of \cite{FeffermanPhong1983}), joining them.

\item \label{A.6} The metric space $(\mathbb{M},d)$ is complete with respect to the intrinsic distance defined by
\begin{equation}\label{distance1}
d(x,y) := \sup \left\lbrace | f(x)-f(y)| : f\in C^\infty(\mathbb{M}), \left\Vert \Gamma\left(f\right)\right\Vert_{\infty}\leqslant 1 \right\rbrace,
\end{equation}
for all $x, y\in\mathbb{M}$ and where we define $\left\Vert g\right\Vert_{\infty}=\operatorname{ess}\sup_{\mathbb{M}}\left|g\right|$.
\end{enumerate}

Note that by \cite[Lemma 5.29]{CarlenKusuokaStroock1987} and \cite[Equation (2.4)]{BaudoinBonnefontGarofalo2014} we know that Assumption \ref{A.5} implies that $d$ is indeed a metric  on $\mathbb{M}$.

As a consequence of Assumption  \ref{A.6}, the  operator $L$ is essentially self-adjoint on $C_{c}^{\infty}\left(\mathbb{M}\right)$ ( e.g. \cite[Proposition 1.20, Proposition 1.21]{BaudoinLevico2018}).  Thus, the pre-Dirichlet form $\mathcal{E}(f,g)$ defined on $C_{c}^{\infty}\left(\mathbb{M}\right)$ by
\[
\mathcal{E}(f,g)=\int_{\mathbb{M}} \Gamma(f,g)d\mu,
\]
has a unique closure as a Dirichlet form, and the generator of this Dirichlet form is the Friedrichs  extension of $L$. We define the Sobolev space $W^{1,2}\left(\mathbb{M}\right)$ to be the domain $\mathcal{D}$ of $\mathcal{E}$ with the norm on $W^{1,2}\left(\mathbb{M}\right)$  given by
\[
\| f \|_{W^{1,2}\left(\mathbb{M}\right)} = \sqrt{\left\Vert f \right\Vert_{2}^{2}+\mathcal{E}\left(f, f\right)}.
\]
It is a consequence of Assumption \ref{A.3} that the metric measure space $(\M, d, \mu)$ satisfies the volume doubling property and supports a scale invariant $L^{2}$-Poincar\'e inequality on metric balls (see \cite{BaudoinBonnefontGarofalo2014}). In particular, by \cite[Chapter 8 ]{HKST15} locally Lipschitz continuous functions form a dense subclass in $W^{1,2}\left(\mathbb{M}\right)$.

For an open set $U\subset\mathbb{M}$, one can define the local Sobolev space $W_{\text{loc}}^{1,2}\left(U\right)$ to be the space of all functions $f$ such that for any compact set $K\subset U$ there exists $F\in\mathcal{D}$ satisfying $f=F$ a.e. on $K$. For
each $p\geqslant 2$, we define $W^{1, p}(U)$ to be space of functions $f\in W_{\text{loc}}^{1,2}(U)$
satisfying $f,\sqrt{\Gamma(f)}\in L^{p}\left(U\right)$.

\subsection{Examples}\label{sec:examples}
We remark that so far the approach has been purely analytical as we have not mentioned any geometric structure of these sub-Riemannian manifolds. In fact $\mathbb{M}$ and $L$ are rather general, even though we have sub-Riemannian manifolds in mind for $\mathbb{M}$.

\subsubsection{Sum of squares operators}
We start by recalling a natural setting where assumption \ref{A.5} is satisfied. Let us consider $L$ that are \emph{sums of squares} operators in the form of
\begin{equation}\label{sumofsquares}
L=\sum_{i=1}^{m}X_i^2+X_0,
\end{equation}
where the $X_i$ are $C^\infty$ vector fields. We refer the reader to \cite{GordinaLaetsch2016a} for more details on operators of the form given in \eqref{sumofsquares} in the context of sub-Riemannian manifolds. Consider the following assumption.

\begin{assumption}(H\"{o}rmander's condition)\label{HCassump}
We will say  that $L$ satisfies \emph{H\"{o}rmander's (bracket generating) condition} if the vector fields $\left\lbrace X_1,\dots,X_m \right\rbrace$ with their Lie brackets span the tangent space $T_{x}\mathbb{M}$ at every point $x \in \mathbb{M}$.
\end{assumption}
H\"{o}rmander's condition guarantees analytic and topological properties such as hypoellipticity of $L$ and topological properties of $\mathbb{M}$. The Chow-Rashevski theorem says that H\"{o}rmander's condition is sufficient to ensure that any two points in $\mathbb{M}$ can be connected by a finite length sub-unit curve. Thus, operators $L$ of the form \eqref{sumofsquares} that satisfy H\"{o}rmander's condition  automatically satisfy assumption \ref{A.5}.

\subsubsection{Other examples}
We note that the assumptions \ref{A.1}-\ref{A.6} are satisfied for a large class of sub-Riemannian manifolds. Such a class includes all Sasakian manifolds whose horizontal Webster-Tanaka-Ricci curvature is bounded below, a wide subclasses of principal bundles over Riemannian manifolds whose Ricci curvature is bounded below, and Carnot groups of step 2. We remark that in general we do not know if  $CD\left(\rho_{1},\rho_{2},\kappa,d\right)$  is satisfied for Carnot groups of an arbitrary step. This shows the need to treat Carnot groups separately in Theorem \ref{CarnotCY}. We refer the reader to \cite{BaudoinGarofalo2017} for a comprehensive treatment on sub-Riemannian manifolds satisfying assumptions \ref{A.1}-\ref{A.6}.

\subsection{The Cheng-Yau estimate}
Recall that we say a function $u:\mathbb{M}\to \mathbb{R}$ is \emph{harmonic} in a domain $D\subset \mathbb{M}$ if $Lu=0$ on $D\subset \mathbb{M}$. We can now state the main result of this section.

\begin{theorem}\label{ChengYau-Transverse}
Suppose assumptions \ref{A.1}-\ref{A.6} hold for $\mathbb{M}$ and $L$. If $u$ is any positive harmonic function for $L$ in a ball $B\left(x, 2r\right)\subset\mathbb{M}$,
then there exists a constant $C>0$ not dependent on $u,r$
and $x$ such that
\begin{equation}
\sup_{B(x,r)} \sqrt{\Gamma\left(\log u \right)} \leqslant\frac{C}{r}.\label{eq:CY-1}
\end{equation}
Moreover, if $u$ is any positive harmonic function on $\mathbb{M}$,
then $u$ must be equal to a constant.
\end{theorem}
\begin{proof}
In the proof, $C,C_2$ will denote generic positive constants that do not depend
on $u, r$ and $x_0$, whose values might change from line to line. First we check that the assumptions of the results in \cite[Theorem
1.2, Lemma 2.3]{CoulhonJiangKoskelaSikora2017} are satisfied. The results in \cite{CoulhonJiangKoskelaSikora2017}
require the assumption that the underlying space is a non-compact
doubling Dirichlet metric measure space. Doubling is shown in $(1.11)$ of \cite[Theorem
1.5]{BaudoinBonnefontGarofalo2014}. Further, we need $\mathbb{M}$ to satisfy upper Gaussian bounds
on the heat kernel, which is given in \cite[Theorem
4.1]{BaudoinBonnefontGarofalo2014}. Finally
we need $\mathbb{M}$ to support a local $L^{2}$-Poincar\'e inequality
of the form
\begin{equation}
\frac{1}{\mu\left(B(x,r)\right)}\int_{B(x,r)}\left|f-f_{B}\right|d\mu\leqslant C_{2}r\left(\frac{1}{\mu \left(B(x,r)\right)}\int_{B(x,r)}\Gamma\left(f\right)d\mu\right)^{\frac{1}{2}},\label{P2}
\end{equation}
for some $C_{2}>0$, for every ball $B(x,r)$ and each $f\in W^{1,2}\left(B(x,r)\right)$.  Here we used the notation $f_{B}=\mu\left(B\right)^{-1}\int_{B}fd\mu$. To see this,
by (1.12) in \cite[Theorem
1.5]{BaudoinBonnefontGarofalo2014} we have that there exists constant
$C_{2}>0$, depending only on $\rho_{1},\rho_{2},\kappa,d$, for which
one has for every $x\in\mathbb{M}$ and every $r>0$,
\begin{equation}
\int_{B(x,r)}\left|f-f_{B}\right|^{2}\leqslant C_{2}r^{2}\int_{B(x,r)}\Gamma(f)d\mu,\label{WeakPoincare}
\end{equation}
for every $f\in C^{1}\left(\overline{B}(x,r)\right)$. Using Cauchy-Schwarz inequality,
followed by $(\ref{WeakPoincare})$ we have that
\begin{align*}
\frac{1}{\mu \left(B(x,r)\right)}\int_{B(x,r)}\left|f-f_{B}\right|d\mu & \leqslant\frac{1}{\mu \left(B(x,r)\right)}\left(\int_{B(x,r)}\left|f-f_{B}\right|^{2}d\mu\right)^{\frac{1}{2}}\sqrt{\mu \left(B(x,r)\right)}\\
 & \leqslant\frac{1}{\sqrt{\mu \left(B(x,r)\right)}}\left(C_{2}r^{2}\int_{B(x,r)}\Gamma(f)d\mu\right)^{\frac{1}{2}}\\
 & =\sqrt{C_{2}}r\left(\frac{1}{\mu \left(B(x,r)\right)}\int_{B(x,r)}\Gamma(f)d\mu\right)^{\frac{1}{2}}.
\end{align*}
This shows \eqref{P2} holds for all $f\in C^{1}\left(\overline{B}(x,r)\right)$. By a density argument we can show that \eqref{P2} holds for all $f\in W^{1,2}\left(B(x,r)\right)$, as needed.

To finish off the proof, we note that Corollary 3.5 in \cite{BaudoinGarofalo2013} shows
that

\begin{equation}
\left\Vert \sqrt{\Gamma\left(P_{t}f\right)}\right\Vert _{\infty\to\infty}\leqslant\frac{C}{\sqrt{t}},\label{eq:Boundness-1}
\end{equation}
where $C=n\sqrt{\frac{\left(2\kappa+\rho_{2}\right)}{2\rho_{2}}}$.
Using the estimate $(\ref{eq:Boundness-1})$, the rest of the proof
becomes similar to the proof of Theorem \ref{CarnotCY}.
\end{proof}


\begin{acknowledgement}

The authors would like to thank the anonymous referee for their careful review of the paper. P. Mariano is also grateful for helpful and insightful conversations with O. Munteanu.

\end{acknowledgement}

\bibliographystyle{amsplain}
\providecommand{\bysame}{\leavevmode\hbox to3em{\hrulefill}\thinspace}
\providecommand{\MR}{\relax\ifhmode\unskip\space\fi MR }
\providecommand{\MRhref}[2]{%
  \href{http://www.ams.org/mathscinet-getitem?mr=#1}{#2}
}
\providecommand{\href}[2]{#2}

\end{document}